\documentclass[a4paper,12pt]{article}
\usepackage{times, url}
\usepackage{graphicx,amsmath,amssymb,amsthm,amsfonts,bm,float,cite}
\newtheorem{theorem}{Theorem}[section]

\newtheorem{lemma}{Lemma}[section]

\newcommand{\disp}{\displaystyle}
\def\F{{\mathbb F}}
\def\Q{{\mathbb Q}}                               
\def\Z{{\mathbb Z}}                               

\def\s1{\sum_{i = 1}^{j_{\alpha}}({t_i}^3~-~ 3{t_i}{\Delta_i})}
\def\u2{\sum_{m = 1}^{l} {\tilde t_m}^3}
\def\v3{\sum_{j = 1}^{s} 3~{\tilde t'_j}}
\def\w4{\sum_{i=1}^{j_{\alpha}}({t_i}^4~-~4~{t_i}^2~{\Delta_i}~+~2~{\Delta_i^2})}
\def\x5{\sum_{m = 1}^{l} {\tilde t_m}^4}
\def\y6{\sum_{j = 1}^{s} 4~{\tilde {t'_j}}^2{\tilde {\Delta'_j}}}
\def\z7{\sum_{n = 1}^{r} 2~{\tilde {\Delta''_n}}^2}

\def\vpmod{\!\!\!\pmod} 

\begin{document}
\setcounter{page}{1}

\thispagestyle{empty} 

\begin{center}
{\LARGE \bf  Simpler congruences for Jacobi sum $J(1,1)_{49}$ of order 49.\\[4mm] } 
\vspace{8mm}

{\Large \bf Ishrat Jahan Ansari$^1$, Vikas Jadhav$^2$ and Devendra Shirolkar$^3$}
\vspace{3mm}

$^1$ Department of Mathematics, M.C.E. Society's Abeda Inamdar Senior College,\\ Savitribai Phule Pune University\\ 
Research Scholar, Sir Parshurambhau College\\
Pune, Maharashtra, India \\
e-mail: \url{ishrat1984@gmail.com}
\vspace{2mm}

$^2$ Department of Mathematics, Nowrosjee Wadia College\\ 
Pune, Maharashtra, India \\
e-mail: \url{svikasjadhav@gmail.com}
\vspace{2mm}

$^3$ Savitribai Phule Pune University\\
Pune, Maharashtra, India \\
e-mail: \url{dshirolkar@gmail.com},   
\end{center}

{\bf Abstract:} In this paper we determine congruence of Jacobi sums $J(1,1)_{49}$ of order 49 over a field $\F_{p}$. 
We also show that simpler congruences hold for $J(1,1)_{49}$ in the case of artiad and hyperartiad primes.  
\vspace{5mm}

\section{Introduction} \label{sec:Intr}

For a positive integer $e \geq 2$, the Jacobi sums of order $e$ are algebraic integers in the cyclotomic field $\Q(\zeta_{e})$, where $\zeta_{e} = \exp(2\pi i/e)$. 
These are defined in the set up of a finite field $\F_{q}$ of $q =p^{r}$ where $q \equiv 1\,\, \vpmod{e}$, $p$ prime. 
Jacobi sums are important in the theory of cyclotomy and their congruences have been studied by many authors.
Earlier authors \cite{2} obtained congruences in the set up of $\F_{p}$, $p\equiv 1\,\, \vpmod{e}$  and later authors \cite{5} considered $q = p^r \equiv 1\,\, \vpmod{e}$.
\begin{enumerate}
\item It is well known that (\cite{2} and \cite{13}) the Jacobi sums of odd prime order $l$, 
$J(1,j)_{l}\equiv -1\,\,\vpmod{(1-\zeta_{l}}^2)$.\\ This congruence also holds $\vpmod{(1-\zeta_{l}}^3)$ (\cite{7} and \cite{14}).
\item Congruence of Jacobi sums of order $2l$ ($l$ odd prime) were obtained by V. V. Acharya, S. A. Katre \cite{1}. 
They showed that\\ $J(1,n)_{2l}\equiv -\zeta^{m(n+1)}\,\, \vpmod{(1-\zeta_{l}}^2)$. Where $n$ is an odd integer such that $1\leq n\leq 2l-3$ and $m= ind\,\, 2$.
\item A congruence of Jacobi sum $J(1,1)_{9}$ of order 9 was obtained by S.  A.  Katre and Rajwade \cite{8} they showed that\\
$J(1,1)_{9}\equiv -1-(ind\,\, 3)(1-\omega)\,\,\vpmod{(1-\zeta_{9}}^4)$ where $\omega={\zeta_{9}}^3$.
\item Congruences of order $l^2$ ($l$ odd prime) were obtained by Devendra Shirolkar and S. A. Katre. 
Refer to Theorem and Remarks followed (\cite{15}). They showed\\
\begin{equation*}
J(1,n)_{l^2} \equiv 
\begin{cases}
-1 + \displaystyle{\sum_{i=3}^{l}c_{i,n}(\zeta -1)^i} \vpmod{(1-\zeta)^{l + 1}} & \text{if $\gcd (l,n) =1$,}\\
-1  \vpmod{(1-\zeta)^{l + 1}} &\text{if $\gcd (l,n) = l$,}
\end{cases}
\end{equation*}
where $1\leq n\leq l^{2}-1$.
\item If $k$ is a an odd power $>3$ (Refer to \cite{6})\\
$J(i,j)_{k}\equiv -1\,\, \vpmod{(1-\zeta_{k}}^3)$\\
R. J. Evans \cite{5} generalised this result to all $k>2$ by elementary methods getting sharper congruences in some cases especially when $k>8$ is a power of 2.
\end{enumerate}
\section{Preliminaries}
Let $e$ be a positive integer $\ge 2$ and  $q=p^{r}\equiv 1\vpmod{e}$, $p$ prime. Let $\F_{q}$ be a finite field with $q$ elements. Write $p^{r}=q= ef + 1$. Let $\zeta$ be a complex primitive $e$th root of unity. If $\gamma$ is a generator of $\F_{q}^{*}$ then define the multiplicative character $\chi$ : $\F_{q}\rightarrow \Q(\zeta)$ by $\chi(\gamma) =\zeta$, $\chi(0)= 0$.
\noindent Given a generator $\gamma$ of $\F_{q}^{*}$ define the Jacobi sum $J(i,j)_{e}$ by,
$$J(i,j) = J(i,j)_{e} = \sum_{v\in \F_{q}}\chi^{i}(v)\chi^{j}(1+v),\,\, 0\leq i,j\leq e-1.$$
Here $\chi^{0}(0) = 0$. Also, $i$ and $j$ can be considered modulo $e$, with the understanding that $\chi^{i}(0) = 0$ for any integer $i$. Note that $J(i,j)_{e} \in \Z[\zeta]$, the ring of integers of $\Q(\zeta).$\\
\indent A variation of the Jacobi sum is defined as,
$$J(\chi^{i},\chi^{j})_{e} = \sum_{v\in \F_{q}}\chi^{i}(v)\chi^{j}(1-v),\,\,0\leq i, j\leq e-1.$$
Observe that $J(i,j)_{e}$= $\chi^{i}(-1)J(\chi^{i},\chi^{j})_{e}$. 
When $q = 2^r$, $\chi^{i}(-1) = \chi^{i}(1) =1$ and both the Jacobi sums coincide. 
Otherwise, $\chi^{i}(-1) = (-1)^{if}$ and hence the two Jacobi sums differ at most in sign.
 For multiplicative characters $\chi$ and $\psi$ on $\F_{q}$, $J(\chi,\psi)$ can be
 analogously defined.

In the following theorem, we state some standard results about Jacobi sums.
\begin{theorem}\label{thm:theorem1}(Elementary properties of Jacobi sums)
\begin{enumerate}
\item[1)] If $i$ and $j$ are congruent to $0$ modulo $e$ then $J(\chi^{i},\chi^{j})_{e} = q - 2$. 
\item[2)] If exactly one of $i$ and $j$ is congruent to $0$ modulo $e$, then $J(\chi^{i},\chi^{j})_{e} = -1$. 
\item[3)] If $i$ is nonzero modulo $e$ and $i+j$ is congruent to $0$ modulo $e$ then  $J(\chi^{i},\chi^{j})_{e} = -\chi^{i}(-1)$.
\item[4)] $J(\chi^{i},\chi^{j})_{e} = J(\chi^{j},\chi^{i})_{e} = \chi^{i}(-1)J(\chi^{-i-j}, \chi^{i})_{e}$.
\item[5)] If $e$ does not divide $i,j$ and $i + j$ then $\mid J(\chi^{i},\chi^{j})_{e}\mid  =  \sqrt{q}$.
\end{enumerate}
\end{theorem}
\begin{proof} See \cite {2} for $q = p$ case and \cite{16} for $q = p^r$.\end{proof}
\noindent \textbf{Remark}: If $f$ is even or $q =2^r$ then $J(i,j)_e = J(\chi^i, \chi^j)_e$, so (4) gives $J(i,j)_e = J(j,i)_e = J(-i-j,j)_e = J(j,-i-j)_e = J(-i-j,i)_e = J(i,-i-j)_e$. In particular $J(i,i)_e = J(-2i,i)_e = J(i, -2i)_e$.
\section{Cyclotomy}
Let $\gamma$, $\zeta$ and $\chi$ be as in Section 2. For $ 0 \leq i, j \leq e - 1$ ($i,j\,\,\vpmod e$), define the $e^{2}$ cyclotomic numbers $(i,j)_{e}$ by $(i,j)_{e}$ = Card.$(X_{ij})$ where
\begin{eqnarray*} 
X_{i j}&=&\{v \in \F_{q}\,\, |\,\, \chi(v) = \zeta^{i}, \chi(v + 1) = \zeta^{j}\}\\
&=& \{v \in \F_{q}- \{0, -1\}\,\, | \,\, {\rm{ind}}_\gamma v \equiv i\vpmod{e},\,\, \textrm{ind}_\gamma (v + 1) \equiv j \vpmod{e}\}.
\end{eqnarray*}  
We state below some basic properties of the cyclotomic numbers. (See \cite{3} for $q=p$, \cite{16}). For $q =p^r$,
\begin{eqnarray*}
(i,j)_{e}& =& (i',j')_{e}\,\,\,\,\,\,{\rm{if}}\,\, i\equiv i' \,\,\,{\rm{and}}\,\,\,\, j \equiv j'\vpmod{e}.\\
(i,j)_{e} &=& (e-i,j-i)_{e}.\\
&=& \begin{cases}
(j,i)_{e},\,\,\,\,\,\,\,{\rm{if}}\,\,f \,\,{\rm{is \,\,even}} \,\,{\rm{or}}\,\,q = 2^r,\\
(j+\frac{1}{2}e, i+\frac{1}{2}e)_{e},\,\,\,\,\,\,\,\,{\rm{otherwise.}}
\end{cases}
\end{eqnarray*} 
Thus if $f$ is even or $q =2^r$, $r\geq 2$ then
\begin{eqnarray}\label{eq:1}
(i,j)_e&=&(j,i)_e=(i-j,-j)_e =(j-i,-i)_e\notag\\
&=&(-i,j-i)_e=(-j,i-j)_e.
\end{eqnarray} 
The $e^{2}$ Jacobi sums and the $e^2$ cyclotomic numbers are related by
\begin{eqnarray}\label{eq:2}
\sum_{i}\sum_{j} \zeta^{-(ai+bj)}J(i,j)_{e} = e^2(a,b)_e ,
\end{eqnarray}
\noindent and
\begin{eqnarray}\label{eq:3}
\sum_{i}\sum_{j}(i,j)_e\zeta^{ai+bj} = J(a,b)_{e}.
\end{eqnarray}
Jacobi sums and cyclotomic numbers are related to Dickson-Hurwitz sums. These are defined for $i, j \pmod{e}$ by  (for $q = p$, see \cite{2})
\begin{eqnarray}\label{eq:4}
 B(i,j) = B(i,j)_{e}= \sum_{h=0}^{e-1}(h,i-jh)_e.
\end{eqnarray}
\noindent They satisfy the relation $B(i,j)_{e} = B(i,e-j-i)_{e}$. Also,
\begin{equation}\label{eq:5}
B(i,0)_{e} =
\begin{cases}
f-1 & \text{if $i = 0$,}\\
f & \text{if $1\leq i \leq{e-1}$.}
\end{cases}
\end{equation}
and \begin{equation}\label{eq:6}
\sum_{i=0}^{e-1}B(i,j)_{e} = q-2.\end{equation}
\noindent Dickson-Hurwitz sums and Jacobi sums $J(\chi,\chi^{j})_{e}$ are related by (for $q = p$, see \cite{2})
\begin{eqnarray}\label{eq:7}
\chi^{j}(-1)J(\chi,\chi^{j})_{e} = \chi^{j}(-1)\chi(-1)J(1,j)_{e} = \sum_{i=0}^{e-1}B(i,j)_{e}\zeta^i.
\end{eqnarray}
Hence if $f$ is even or $q =2^r$ then $J(1,j)_{e} = \displaystyle{\sum_{i=0}^{e-1}B(i,j)_{e}\zeta^i.}$
\section{Congruences of Jacobi sums $J(1,n)_{l^2}$ of order $l^{2}$}
\noindent The determining congruences of Jacobi sums $J(1,\,n)_{l^2}$ of order $l^{2}$ have been studied by Devendra Shirolkar and S.A. Katre \cite{15}. This congruence is in the terms of linear combination of cyclotomic numbers of order $l$. Their work generalises the work of R.J. Evans \cite{5}. We state their important result for ready reference.
\begin{lemma}\label{lem:lemma1}
Let $l>3$ be a prime and  $1\leq n \leq l^2 -1$. Write $n =dl +n'$ where $1\leq n' \leq l-1$  . For $1\leq h \leq l-1$, let\begin{eqnarray*}
\lambda_{h} = \lambda_{h}(n) &=& \left[\frac{n'h}{l}\right] + \left[\frac{-h(n' + 1)}{l}\right], \end{eqnarray*} and $1\leq h,k \leq l-1 $, $h \neq k$, let \begin{eqnarray*}\lambda_{h,k} = \lambda_{h,k}(n)&=& \left[\frac{h + n'k}{l}\right] + \left[\frac{k + n'h}{l}\right] + \left[\frac{n'k -h(n'+1)}{l}\right]\\ &&+ \left[\frac{n'h - k(n'+ 1)}{l}\right] + \left[\frac{k - h(n' + 1)}{l}\right] + \left[\frac{h -k(n'+1)}{l}\right].
\end{eqnarray*} For a given $n$, $\lambda_{h,k}$ depends only on the class of six elements (cf. \eqref{eq:1}) to which $(h,k)_l$ belongs. Define \begin{eqnarray*}S(n):= \sum_{t = 0}^{l-1}\sum_{j= 0}^{l-1}tB(lt + j, n)_{l^2}.\end{eqnarray*} Then 
\begin{eqnarray*}
S(n) \equiv 
\disp{\sum_{h = 1}^{l-1}}\lambda_{h} (h,0)_{l} + \disp{\sum_{c}}\lambda_{h,k}(h,k)_{l} \pmod l
\end{eqnarray*} 
where $\disp{\sum_{c}}$ is taken over a set of representatives of classes of six elements of cyclotomic numbers 
of order $l$, obtained with respect to \eqref{eq:1}. Furthermore $S(n)\equiv 0 \pmod l$ if gcd $(l,n) = l.$
\end{lemma}
\begin{proof}
Refer \cite{15}.
\end{proof}
\begin{theorem}\label{thm:theorem2}
Let $l > 3$ be a prime and $p^r = q \equiv 1\,\, \vpmod{l^2}$. If $ 1 \leq n \leq l^2 -1$, then a (determining) congruence for $J(1,n)_{l^2}$ for a finite field $\F_{q}$ is given by
\begin{equation*}
J(1,n)_{l^2} \equiv 
\begin{cases}
-1 + \displaystyle{\sum_{i=3}^{l}c_{i,n}(\zeta -1)^i} \vpmod{(1-\zeta)^{l + 1}} & \text{if $\gcd (l,n) =1$,}\\
-1  \vpmod{(1-\zeta)^{l + 1}} &\text{if $\gcd (l,n) = l$,}
\end{cases}
\end{equation*}
where for $3 \leq i \leq l-1$, $\displaystyle{c_{i,n}= \sum_{u=i}^{l-1}{u\choose i}B(u,n')_{l}}$ and $c_{l,n} = S(n)$ is given by Lemma \ref{lem:lemma1}.
\end{theorem}
\begin{proof}
Refer \cite{15}.
\end{proof}
\section{Cyclotomic numbers of order 7}
There are 49 cyclotomic numbers if order 7. Out of these only twelve distinct cyclotomic numbers of order 7 are sufficient to determine the remaining (See equation(1)). 
If $p\equiv 1\,\,\vpmod7$ the Diophantine system of Leonhard and Williams is given by (see \cite{10}).
There are 49 cyclotomic numbers of order 7. Out of these only twelve distinct cyclotomic numbers of order 7 are sufficient to determine the remaining.
If $p\equiv 1\,\,\vpmod7$ the Diophantine system of Leonard and Williams is given by (see \cite{10}).\\

	$\begin{aligned}
		72p&=2{x_{1}}^2+42({x_{2}}^2+{x_{3}}^2+{x_{4}}^2)+343({x_{5}}^2+3{x_{6}}^2),\\
		&12{x_{2}}^2-12{x_{4}}^2+147{x_{5}}^2-441{x_{6}}^2+56x_{1}x_{6}+\\
		&24x_{2}x_{3}-24x_{2}x_{4}+48x_{3}x_{4}+98x_{5}x_{6}=0,\\
		&12{x_{5}}^2-12{x_{4}}^2+49{x_{5}}^2-147{x_{6}}^2+28x_{1}x_{5}\\
		&+28x_{1}x_{6}+48x_{2}x_{3}+24x_{3}x_{4}+490x_{5}x_{6}=0.
	\end{aligned}$\\

$x_{1}\equiv 1\vpmod 7 $ has six non-trivial solutions in addition to the two trivial solutions\\ $(-6t, \pm2u,\pm2u, \mp2u, 0, 0)$
where $t$ and $u$ are given by $p=t^2+7u^2$, $t\equiv 1\vpmod7$.
If $X_{1}=(x_{1},x_{2},x_{3}, x_{4}, x_{5}, x_{6})$ is one of the non-trivial solutions then the other five non-trivial solutions are: (see \cite{10}.)\\
\begin{center}
$\begin{aligned}
	X_{2}=(x_{1}, x_{3},-x_{4}, x_{2},-\frac{1}{2}(x_{5}+3x_{6}),\frac{1}{2}(x_{5}-x_{6})),\\
	X_{3}=(x_{1}, x_{4},-x_{2}, x_{3},-\frac{1}{2}(x_{5}-3x_{6}),-\frac{1}{2}(x_{5}+x_{6})),\\
	X_{4}=(x_{1}, -x_{4},x_{2}, -x_{3},-\frac{1}{2}(x_{5}-3x_{6}),-\frac{1}{2}(x_{5}+x_{6})),\\
	X_{5}=(x_{1}, -x_{3},x_{4}, x_{2},-\frac{1}{2}(x_{5}+3x_{6}),\frac{1}{2}(x_{5}-x_{6})),\\
	X_{6}=(x_{1},-x_{2},-x_{3}, -x_{4}, x_{5}, x_{6}).
\end{aligned}$
\end{center}
For a suitable choice of solution of the above Diophantine system, the cyclotomic numbers of order 7 are given by (see \cite{11}).\\
\begin{eqnarray}\label{eq:8}
\begin{aligned}
49(0,0)&=p-20-12t+3x_{1}\\
588(0,1)&=12p-72+24t+168u-6x_{1}+84x_{2}\\
& -42x_{3}+147x_{4}+147x_{6}\\
588(0,2)&=12p-72+24t+168u-6x_{1}+84x_{3}\\
&+42x_{4}-294x_{6}\\
588(0,3)&=12p-72+24t-168u-6x_{1}+42x_{2}\\
&+84x_{4}-147x_{5}+147x_{6}\\
588(0,4)&=12p-72+24t+168u-6x_{1}-42x_{2}\\
&-84x_{4}-147x_{5}+147x_{6}\\
588(0,5)&=12p-72+24t-168u-6x_{1}-84x_{3}\\
&-42x_{4}-294x_{6}\\
588(0,6)&=12p-72+24t-168u-6x_{1}-84x_{2}\\
&+42x_{3}+147x_{5}+147x_{6}\\
588(1,2)&=12p+12+24t+8x_{1}-196x_{5}\\
588(1,3)&=12p+12-60t-84u-6x_{1}+42x_{2}\\
&+42x_{3}-42x_{4}\\
588(1,4)&=12p+12+24t+8x_{1}+98x_{5}-294x_{6}\\
588(1,5)&=12p+12-60t+84u-6x_{1}-42x_{2}\\
&-42x_{3}+42x_{4}\\
588(2,4)&=12p+12+24t+8x_{1}+98x_{5}+294x_{6}
\end{aligned}
\end{eqnarray}
Also, if $J(1,1)_{7}=\displaystyle{\sum_{i =0}^{6}c_{i}\zeta^i}=\displaystyle{\sum_{i =0}^{6}B(i,1)_{7}{\zeta}^{i}}$ is Jacobi sums of order 7 then the integers $c_{1},c_{2}\dots c_{6}$ are given by (see \cite{11})
\begin{eqnarray}\label{eq:9}
\begin{aligned}
12c_{1}&=-2x_{1}+6x_{2}+7x_{5}+21x_{6}\\
12c_{2}&=-2x_{1}+6x_{3}+7x_{5}-21x_{6}\\
12c_{3}&=-2x_{1}+6x_{4}-14x_{5}\\
12c_{4}&=-2x_{1}-6x_{4}-14x_{5}\\
12c_{5}&=-2x_{1}-6x_{3}+7x_{5}-21x_{6}\\
12c_{6}&=-2x_{1}-6x_{2}+7x_{5}+21x_{6}
\end{aligned}
\end{eqnarray}
\section {Congruences of Jacobi sum $J(1,1)_{49}$ of order 49}
 Let $p\equiv 1\,\,\vpmod {49}$ be a prime and and $\zeta$ be primitive 49th root of unity in $\mathbb{Q}(\zeta)$ then from Theorem (\ref{thm:theorem2}) the determining congruences for Jacobi sum $J(1,1)_{49}$ of order 49 are given as:
$$J(1,1)_{49}\equiv -1 + \displaystyle{\sum_{i=3}^{7}c_{i,1}(\zeta -1)^i} \vpmod{(1-\zeta)^{8}}$$
where $c_{i,1}$ are defined in Theorem (\ref{thm:theorem2}).
Let $q=7f+1$ ($f$ even) be a prime. Then the Jacobi sum $J(1,1)_{7}$ in terms of the Dickson-Hurwitz sums $B(i,1)_{7}$, ($0\leq i\leq 6$) is given in \eqref{eq:7}. 
These Dickson-Hurwitz sums of order 7 in terms of solutions of the diophantine system are given by P. A. Leonhard and K. S. Williams (see \cite{11})
\begin{eqnarray}\label{eq:10}
\begin{aligned}
84B(0,1)_{7}&=12x_{1}+12p-24\\
84B(1,1)_{7}&=-2x_{1}+42x_{2}+49x_{5}+147x_{6}+12p-24\\
84B(2,1)_{7}&=-2x_{1}+42x_{3}+49x_{5}-147x_{6}+12p-24\\
84B(3,1)_{7}&=-2x_{1}+42x_{4}-98x_{5}+12p-24\\
84B(4,1)_{7}&=-2x_{1}-42x_{4}-98x_{5}+12p-24\\
84B(5,1)_{7}&=-2x_{1}-42x_{3}+49x_{5}-147x_{6}+12p-24\\
84B(6,1)_{7}&=-2x_{1}-42x_{3}+49x_{5}+147x_{6}+12p-24
\end{aligned}
\end{eqnarray}
Therefore,
\begin{eqnarray}\label{eq:11}
\begin{aligned}
c_{1,1}&=(\frac{6p-x_{1}-12}{2})-(\frac{x_{4}+3x_{3}+5x_{2}}{2})\\
c_{2,1}&=(\frac{5}{3})(\frac{6p-x_{1}-12}{2})-3(\frac{x_{4}+3x_{3}+5x_{2}}{2})+(\frac{28x_{5}+42x_{6}}{6})\\
c_{3,1}&=(\frac{5}{3})(\frac{6p-x_{1}-12}{2})-3(\frac{3x_{4}+10x_{3}+20x_{2}}{2})+(\frac{105x_{6}+70x_{5}}{6})\\
c_{4,1}&=(\frac{6p-x_{1}-12}{2})-(\frac{x_{4}-5x_{3}-15x_{2}}{2})+(\frac{35x_{6}+21x_{5}}{2})\\
c_{5,1}&=(\frac{2}{6})(\frac{6p-x_{1}-12}{2})-(\frac{x_{3}+6x_{2}}{2})+(\frac{105x_{6}+49x_{5}}{12})\\
c_{6,1}&=(\frac{2}{42})(\frac{6p-x_{1}-12}{2})-(\frac{9x_{3}+14x_{2}}{28})+(\frac{21x_{6}+7x_{5}}{12})
\end{aligned}
\end{eqnarray}
Using equation \eqref{eq:8} and lemma (\ref{lem:lemma1}).
\begin{equation*}
c_{7,1}=-(\frac{2}{14})(\frac{6p-x_{1}-12}{2})+(\frac{2}{14})(\frac{3x_{3}+5x_{2}}{2})+(\frac{7x_{5}-5x_{4}}{28}).
\end{equation*}
We observe that $x_{1}\equiv 1\,\,\vpmod 7 $ and $p\equiv 1\,\,\vpmod 7$ therefore, $\displaystyle{\frac{6p-x_{1}-12}{2}\equiv 0\,\,\vpmod 7}$.
\subsection{Congruence of Jacobi sum $J(1,1)_{49}$ for artiad and hyperartiad primes.}
 Lloyed Tanner came across some special primes while studying Jacobi sums in the field of $5^{th}$ root of unity over the field $\F_{p}$ where $p=10f+1$.
 He observed that, when Jacobi sums corresponding to these primes were expanded so that the sum of their coefficients is -1, he saw that the  coefficients of Jacobi sums are congruent modulo 5. He called these primes as artiad primes.
 Later in 1985 Emma Lehmer gave a characterization of such primes (Refer \cite{9}). \\
The prime $p=14s+1$ for which all solutions of congruence $x^{3}+x^{2}-2x-1\equiv 0(modp)$ are seventh power residues is an artiad prime. 
Spetic hyperartiad primes are septic artiad primes for which 7 is a seventh power residue.
 In this section, we provide another characterization of such prime $p$, $p=14s+1$ in terms of $x_{1}, x_{2},\ldots, x_{6}$ and 
 show that for such primes, simpler congruences hold for Jacobi sums of order 49.
\begin{lemma}\label{lem:lemma2}
 $p=14s+1$ is an artiad prime if and only if $x_{2}\equiv x_{3}\equiv x_{4}\equiv 0\,\,\vpmod 7$.
\end{lemma}
\begin{proof} Let $p=14s+1$ be an artiad prime. Then from the work of Emma Lehmer (Refer \cite{9}, Section 5 Theorem 5)
$c_{k}\equiv c_{7-k}\,\,\vpmod 7$, $k=1,2,3$. Using \eqref{eq:9} we get,
$x_{2}\equiv x_{3}\equiv x_{4}\equiv 0\,\,\vpmod 7$.\\
\hspace*{5mm} Conversely, suppose $p=14s+1$ with $x_{2}\equiv x_{3}\equiv x_{4}\equiv 0\,\,\vpmod 7$ and $x_{1}\equiv 1\vpmod 7$. From \eqref{eq:9} $c_{k}\equiv c_{7-k}\,\,\vpmod 7$, $k=1,2,3$.
Therefore, $p$ is an artiad prime.
\end{proof} 
\begin{lemma}\label{lem:lemma3}
$p=14s+1$ is a hyperartiad prime if and only if $p$ is an artiad prime and for a generator $\gamma$ of ${\F_{p}}^{*}$ $ind_{\gamma}7\equiv 0\,\,\vpmod 7$.
\end{lemma}
\begin{proof}
J.B. Muskat has given the expression  for $ind_{\gamma}7$ in terms of cyclotomic numbers of order 7 as ( see \cite{12} Section 1, Theorem 1), 
\begin{equation}\label{eq:12}
\displaystyle{ind_{\gamma}7\equiv(\frac{p-1}{2})-\sum_{h=0}^{6}(h,0)_{7}h\vpmod 7}.
\end{equation}\\
Let $p$ be a hyperartiad prime (Hence artiad as well.). Then $(0,h)_{7}\equiv(0,7-h)_{7}\,\,\vpmod7$ (see the work of Emma Lehmer \cite{9} Section 5, Theorem 6). Hence 
by \eqref{eq:17} we get , $ind_{\gamma}7\equiv 0\,\,\vpmod 7$.\\
\indent Conversely, suppose $p=14s+1$ is an artiad prime and $ind_{\gamma}7\equiv 0\,\,\vpmod 7$ then using \eqref{eq:12} we get,
$\displaystyle{\sum_{h=0}^{e-1}(h,0)_{7}h\equiv 0\vpmod 7}$. Therefore, from the work of Emma Lehmer (Refer (\cite{15} Theorem 6 equation (29)). Hence $(0,h)_{7}\equiv(0,7-h)_{7}\,\,\vpmod7$ for $h = 1,2,3$.
\end{proof}
\begin{lemma}\label{lem:lemma4}
 $p=14s+1$ is an artiad prime if and only if $c_{1,1}\equiv c_{2,1}\equiv c_{3,1}\equiv c_{4,1}\equiv c_{5,1}\equiv 0\,\,\vpmod 7$, $12c_{6,1}\equiv (\frac{4}{7})\frac{6p-x_{1}-12}{2}\,\,\vpmod 7$, $4c_{7,1}-4ind_{\gamma}7\equiv -12c_{6,1}+x_{5}\,\,\vpmod 7$.
\end{lemma}
\begin{proof}
Let $p=14s+1$ be an artiad prime. Then by Lemma (\ref{lem:lemma2}) and equation \eqref{eq:11}, we get $c_{1,1}\equiv c_{2,1}\equiv c_{3,1}\equiv c_{4,1}\equiv c_{5,1}\equiv 0\,\,\vpmod 7$
(Being $x_{1}\equiv 1\,\,\vpmod 7$).\\
From equation \eqref{eq:11} $$12c_{6,1}=\frac{12p-2x_{1}-24}{7}+(-21x_{6}+7x_{5}-x_{2})$$
Hence, $$12c_{6,1}\equiv\frac{4}{7}(\frac{6p-x_{1}-12}{2})\vpmod 7.$$
Now we have $28ind_{\gamma}7\equiv x_{2}-19x_{3}-18x_{4}\,\,\vpmod {49}$ (see \cite{16} Corollary 2 equation 8).\\
Hence $$4ind_{\gamma}7\equiv(\frac{x_{2}-19x_{3}-18x_{4}}{7})\vpmod 7.$$\\
Using equation \eqref{eq:11} $$28c_{7,1}=-(12p-7x_{5}+5x_{4}-6x_{3}-10x_{2}-2x_{1}-24).$$\\
Therefore, $$28c_{7,1}-28ind_{\gamma}7\equiv-12p+2x_{1}+24+9x_{2}+25x_{3}+13x_{4}+7x_{5}\vpmod {49}$$ and we get 
$$4c_{7,1}-4ind_{\gamma}7\equiv-12c_{6,1}+x_{5}\vpmod 7.$$
\indent Conversely, suppose $c_{1,1}\equiv c_{2,1}\equiv c_{3,1}\equiv c_{4,1}\equiv c_{5,1}\equiv 0\,\,\vpmod 7$,
   $12c_{6,1}\equiv\frac{4}{7}(\frac{6p-x_{1}-12}{2})\,\,\vpmod 7$, $4c_{7,1}-4ind_{\gamma}7\equiv-12c_{6,1}+x_{5}\,\,\vpmod 7$\\
Using equation \eqref{eq:11} we get
\begin{equation*}\label{eq:16}
  4x_{4}+x_{3}+3x_{2}\equiv 0\vpmod7
  \end{equation*}
  \begin{equation*}\label{eq:17}
  5x_{4}+5x_{3}-4x_{2}\equiv 0\vpmod7
  \end{equation*}
  \begin{equation*}\label{eq:18}
  4x_{4}+5x_{3}-x_{2}\equiv 0\vpmod7
  \end{equation*}
  \begin{equation*}\label{eq:19}
   x_{3}\equiv 6x_{2}\vpmod7
\end{equation*}
Hence, $x_{2}\equiv x_{3}\equiv x_{4}\equiv 0\,\,\vpmod 7$. Therefore, using Lemma (\ref{lem:lemma2}) $p$ is an artiad prime.
\end{proof}
\begin{lemma}\label{lem:lemma5}
$p=14s+1$ is a hyperartiad prime if and only if $c_{1,1}\equiv 0\,\,\vpmod 7$, $c_{2,1}\equiv 0\,\,\vpmod 7$, $c_{3,1}\equiv 0\,\,\vpmod 7$, $c_{4,1}\equiv 0\,\,\vpmod 7$, $c_{5,1}\equiv 0\,\, \vpmod 7$,
$12c_{6,1}\equiv (\frac{4}{7})\frac{6p-x_{1}-12}{2}\,\,\vpmod 7$, $4c_{7,1}\equiv -12c_{6,1}+x_{5}\,\,\vpmod 7$.
\end{lemma}
\begin{proof}
Apply Lemma (\ref{lem:lemma2}) and Lemma (\ref{lem:lemma3}). 
\end{proof}
\begin{theorem}\label{thm:theorem3}
\text{(1)} $p\equiv 1\,\,\vpmod 7$ is an artiad prime if and only if\\
$J(1,1)_{49} \equiv -1+c_{6,1}(\zeta -1)^6+(-3c_{6,1}+ind_{\gamma}7+2x_{5})(\zeta-1)^{7}\,\,\vpmod{(1-\zeta)^{8}}$.\\
\text{(2)} $p$ is a hyperartiad prime if and only if\\
$J(1,1)_{49} \equiv -1+c_{6,1}(\zeta -1)^6+(-3c_{6,1}+2x_{5})(\zeta-1)^7\,\,\vpmod{(1-\zeta)^{8}}$.
\end{theorem}
\begin{proof}
(1) We have $J(1,1)_{49}\equiv -1 + \displaystyle{\sum_{i=3}^{7}c_{i,1}(\zeta -1)^i}\,\, \vpmod{(1-\zeta)^{8}}$.\\
 Let $p$ be an artiad prime then by Lemma (\ref{lem:lemma4}) $c_{1,1}\equiv c_{2,1}\equiv c_{3,1}\equiv c_{4,1}\equiv c_{5,1}\equiv 0\,\,\vpmod 7$.\\
 Therefore, \begin{equation*}\label{eq:20}
 J(1,1)_{49}\equiv -1 +c_{6,1}(\zeta -1)^6+c_{7,1}(\zeta -1)^7 \vpmod{(1-\zeta)^{8}}.
             \end{equation*}
From Lemma (\ref{lem:lemma4}) $c_{7,1}\equiv-3c_{6,1}+ind_{\gamma}7+2x_{5}\vpmod 7.$
Hence, 
 $$J(1,1)_{49}\equiv -1 +c_{6,1}(\zeta -1)^6+(-3c_{6,1}+ind_{\gamma}7+2x_{5})(\zeta -1)^7 \vpmod{(1-\zeta)^{8}}.$$
\noindent Suppose $J(1,1)_{49}\equiv -1 +c_{6,1}(\zeta -1)^6+(-3c_{6,1}+ind_{\gamma}7+2x_{5})(\zeta -1)^7\,\, \vpmod{(1-\zeta)^{8}}$.
 Therefore, $c_{1,1}\equiv c_{2,1}\equiv c_{3,1} \equiv c_{4,1}\equiv c_{5,1}\equiv 0\,\,\vpmod 7$. Repeating the arguments as in Lemma (\ref{lem:lemma4})
 we get, $x_{2}\equiv x_{3}\equiv x_{4}\equiv 0\,\,\vpmod 7$ and hence $p$ is an artiad prime.\\
(2) If $p$ is a hyperartiad prime, then it is an artiad prime. By part(1) $$J(1,1)_{49}\equiv -1 +c_{6,1}(\zeta -1)^6+(-3c_{6,1}+ind_{\gamma}7+2x_{5})(\zeta -1)^7 \vpmod{(1-\zeta)^{8}}.$$
As $p$ is a hyperartiad prime from Lemma (\ref{lem:lemma3}) $ind_{\gamma}7\equiv 0\vpmod7$. Therefore
\begin{equation*}\label{eq29}
J(1,1)_{49}\equiv -1 +c_{6,1}(\zeta -1)^6+(-3c_{6,1}+2x_{5})(\zeta -1)^7 \vpmod{(1-\zeta)^{8}}.
\end{equation*}
\noindent Suppose $J(1,1)_{49}\equiv -1 +c_{6,1}(\zeta -1)^6+(-3c_{6,1}+2x_{5})(\zeta -1)^7\,\, \vpmod{(1-\zeta)^{8}}$. 
Therefore, $c_{1,1}\equiv c_{2,1}\equiv c_{3,1} \equiv c_{4,1}\equiv c_{5,1}\equiv 0\,\,\vpmod 7$. Repeating the argument as in Lemma (\ref{lem:lemma4}) we get, $x_{2}\equiv x_{3}\equiv x_{4}\equiv 0\,\,\vpmod 7$ and hence $p$ is an artiad prime.
Again by lemma (\ref{lem:lemma4}) $12c_{6,1}\equiv (\frac{4}{7})\frac{6p-x_{1}-12}{2}\,\,\vpmod7$ and $4c_{7,1}-4ind_{\gamma}7\equiv-12c_{6,1}+x_{5}\,\,\vpmod7$.
But as $c_{7,1}\equiv-3c_{6,1}+2x_{5}\,\,\vpmod 7$, hence $ind_{\gamma}7\equiv 0\,\,\vpmod7$. Apply lemma (\ref{lem:lemma5}) $p$ is a hyperartiad prime.
 \end{proof}

\makeatletter
\renewcommand{\@biblabel}[1]{[#1]\hfill}
\makeatother


\begin{thebibliography}{99}



\bibitem{1} Acharya, V. V., \& Katre, S. A., Cyclotomic numbers of orders $2l$, $l$ an odd prime. \textit{Acta Arith.}, 69(1) (1995), 51–74.

\bibitem{2} Dickson, L. E., Cyclotomy and trinomial congruences. \textit{Trans. Amer. Math. Soc.}, 37 (1935), 363–380.

\bibitem{3} Dickson, L. E., Cyclotomy, higher congruences, and Waring's problem. \textit{Amer. J. Math.}, 57 (1935), 391–424.

\bibitem{4} Dickson, L. E., Cyclotomy when $e$ is composite. \textit{Trans. Amer. Math. Soc.}, 38 (1935), 187–200.

\bibitem{5} Evans, R. J., Congruences for Jacobi sums. \textit{J. Number Theory}, 71 (1998), 109–120.

\bibitem{6} Ihara, Y., Profinite braid groups, Galois representations, and complex multiplications. \textit{Ann. Math.}, 123 (1986), 43–106.

\bibitem{7} Iwasawa, K., A note on Jacobi sums. In \textit{Symposia Math.}, Vol. 15 (1975), 447–459. Academic Press, London.

\bibitem{8} Katre, S. A., \& Rajwade, A. R., On the Jacobsthal sum $\phi_{9}(a)$ and the related sum $\psi_{9}(a)$. \textit{Math. Scand.}, 53 (1983), 193–202.

\bibitem{9} Lehmer, E., Artiads characterized. \textit{J. Math. Anal. Appl.}, 15 (1966), 118–131.

\bibitem{10} Leonard, P. A., \& Williams, K. S., A Diophantine system of Dickson. \textit{Rend. Accad. Naz. Lincei}, 56 (1974), 145–250.

\bibitem{11} Leonard, P. A., \& Williams, K. S., The cyclotomic numbers of order 7. \textit{Proc. Amer. Math. Soc.}, 51 (1975), 295–300.

\bibitem{12} Muskat, J. B., On the solvability of $x^{e}\equiv e \pmod p$. \textit{Pacific J. Math.}, 14(1) (1964), 257–260.

\bibitem{13} Parnami, J. C., Agrawal, M. K., \& Rajwade, A. R., Jacobi sums and cyclotomic numbers for a finite field. \textit{Acta Arith.}, 41 (1982), 1–13.

\bibitem{14} Parnami, J. C., Agrawal, M. K., \& Rajwade, A. R., A congruence relation between the coefficients of the Jacobi sum. \textit{Indian J. Pure Appl. Math.}, 12(7) (1981), 804–806.

\bibitem{15} Shirolkar, D., \& Katre, S. A., Jacobi sums and cyclotomic numbers of order $l^{2}$. \textit{Acta Arith.}, (2011), 33–49. 

\bibitem{16} Storer, T., \textit{Cyclotomy and Difference Sets}. Markham, Chicago, 1967.
\end{thebibliography}
\end{document}